\documentclass[]{interact}

\usepackage{epstopdf}
\usepackage[caption=false]{subfig}

\usepackage[numbers,sort&compress]{natbib}
\usepackage{graphicx}%
\usepackage{multirow}%
\usepackage{amsmath,amssymb,amsfonts}%
\usepackage{amsthm}%
\usepackage{mathrsfs}%
\usepackage[title]{appendix}%
\usepackage{xcolor}%
\usepackage{textcomp}%
\usepackage{manyfoot}%
\usepackage{booktabs}%
\usepackage{algorithm}%
\usepackage{algorithmicx}%
\usepackage{algpseudocode}%
\usepackage{listings}%
\usepackage{lineno}
\usepackage{txfonts} 
\usepackage{multirow} 
\usepackage{caption} 
\usepackage{txfonts} 
\usepackage{array} 
\usepackage{booktabs} 
\usepackage{longtable}
\usepackage{bm}
\usepackage{setspace} 
\usepackage{natbib} 
\usepackage{hyperref}
\usepackage{diagbox}
\usepackage{makecell}
\usepackage{longtable}

\bibpunct[, ]{[}{]}{,}{n}{,}{,}
\makeatletter
\def\NAT@def@citea{\def\@citea{\NAT@separator}}
\makeatother

\theoremstyle{plain}
\newtheorem{theorem}{Theorem}[section]

\newtheorem{proposition}[theorem]{Proposition}
\newtheorem{assumption}[theorem]{Assumption}

\theoremstyle{definition}

\theoremstyle{remark}
\newtheorem{remark}{Remark}

\begin{document}


\title{An Efficient Algorithm for Computational Protein Design Problem}

\author{
\name{Yukai Zheng\textsuperscript{a}, Weikun Chen\textsuperscript{b} and Qingna Li\textsuperscript{c}\thanks{CONTACT Qingna Li. Email: qnl@bit.edu.cn}}
\affil{\textsuperscript{a}School of Mathematics and Statistics, Beijing Institute of Technology, Beijing, China; \textsuperscript{b}School of Mathematics and Statistics/Beijing Key Laboratory on MCAACI, Beijing Institute of Technology, Beijing, China; \textsuperscript{c}School of Mathematics and Statistics/Beijing Key Laboratory on MCAACI, Beijing Institute of Technology, Beijing, China}
}

\maketitle

\begin{abstract}
In this paper, we consider the computational protein design (CPD) problem, which is usually modeled as a 0/1 programming and is extremely challenging due to its combinatorial properties. We propose an efficient algorithm for solving it. Specifically, we study the quadratic semi-assignment problem formulation (QSAP) of the CPD problem, and show that it is equivalent to its continuous relaxation problem (RQSAP), in terms of sharing the same optimal objective value. Then, we propose an efficient penalty method to solve the QSAP based on the proposed formulations, which is guaranteed to converge to a global solution of the QSAP under certain conditions. Compared with existing branch-and-bound approaches that suffer from high computational complexity, the proposed algorithm is based on a continuous problem and enjoys a low per-iteration complexity, which makes it particularly suitable for solving large-scale CPD problems. Numerical results on benchmark instances verify the superior performance of our approach over the state-of-the-art branch-and-cut solvers. In particular, the proposed algorithm outperforms the state-of-the-art solvers by three order of magnitude in CPU time in most cases, while it still returns high-quality solutions.
\end{abstract}


\section{Introduction}\label{sec1}

The computational protein design (CPD) problem arises from biology, which attempts to guide the protein design process by producing a set of specific proteins that is not only rich in functional proteins, but also small enough to be evaluated experimentally. In this way, the problem of selecting amino acid sequences to perform a given task can be defined as a computable optimization problem. It is often described as the inverse of the protein folding problem \cite{pabo1983molecular,chiu1998optimizing,yue1992inverse}: the three-dimensional structure of a protein is known, and we need to find the amino acid sequence folded into it \cite{creighton1990protein}.

The challenge of the CPD problem lies in its combinatorial properties over different choices of natural amino acids. The resulting optimization model is usually NP-hard \cite{pierce2002protein,traore2013new}. Existing methods for CPD problems make use of different mathematical models, including probabilistic graphical model \cite{thomas2008protein,gainza2016algorithms}, integer linear programming model \cite{zhu2007mixed,lippow2007progress}, 0/1 quadratic programming model \cite{riazanov2017inverse,forrester2008quadratic} and weighted partial maximum satisfiability problem (MaxSAT) \cite{luo2017ccehc,schiex2014computational}. Various models were proposed in different situations with different scopes. However, due to the exponential complexity, these branch-and-bound approaches cannot solve  large-scale CPD problems.
Therefore, some preprocessing methods were proposed to reduce the problem size and improve the solution efficiency \cite{shah2004preprocessing,allouche2014computational,yanover2007dead}. For example, the dead end elimination (DEE) method \cite{allouche2014computational,yanover2007dead} reduces the problem size by eliminating some selection choices in the combinatorial space which does not contain the optimal solution. Such strategies can speed up the algorithm when sovling the CPD problems \cite{allouche2014computational}, but the worst-case complexity of the algorithm itself has not decreased.  

Our interest in this paper is in the mathematical model for the CPD problem. Note that the CPD problem is essentially an integer programming problem. Among various models for integer programming, assignment models and corresponding algorithms have been widely applied in financial decision making \cite{chen2018outranking}, resources allocation \cite{xian2012application} and especially in solving dynamic traffic problems \cite{florian2008application,lai2018real,szeto2012dynamic}. In \cite{hypergraph}, the authors formulate the hypergraph matching problem as an assignment problem, with a nonlinear objective function. Due to the special structure in hypergraph matching problem, the authors propose a continuous relaxation problem which can also recover the optimal solution of the hypergraph matching problem. The key point of such recovery property lies in the linearity of the objective function with each block of assignment variables. Such favorable property is further explored in \cite{mimo}, where the assignment variables are introduced for Multi-Input-Multi-Output (MIMO) detection problem, and exact recovery result is also established therein. 

Inspired by the work above, we consider the CPD problem as a quadratic semi-assignment problem (QSAP) in this paper. The QSAP enjoys the favorable property as in \cite{hypergraph,mimo}, i.e., the objective function is linear with respect to each block of the assignment variables. With this property, the continuous relaxation problem can be proved to recover the global optimal solution of the QSAP. Compared wth existing branch-and-bound approaches that suffer from high computational complexity, the proposed algoirthm is based on a continuous problem and enjoys a low per-iteration complexity, which makes it particularly suitable for solving large-scale CPD problems. Numerical results verify the efficiency of the proposed algorithm.

The rest of this paper is organized as follows. In Section 2, we formulate the CPD problem. In Section 3, we study an equivalent continuous relaxation problem of the CPD problem. In Section 4, we propose a quadratic penalty method to solve the relaxation problem. In Section 5, we report the numerical results. Final conclusions are made in Section 6.

\section{Problem Formulation}\label{sec2}

\begin{figure}
\centering
\includegraphics[width=0.6\linewidth]{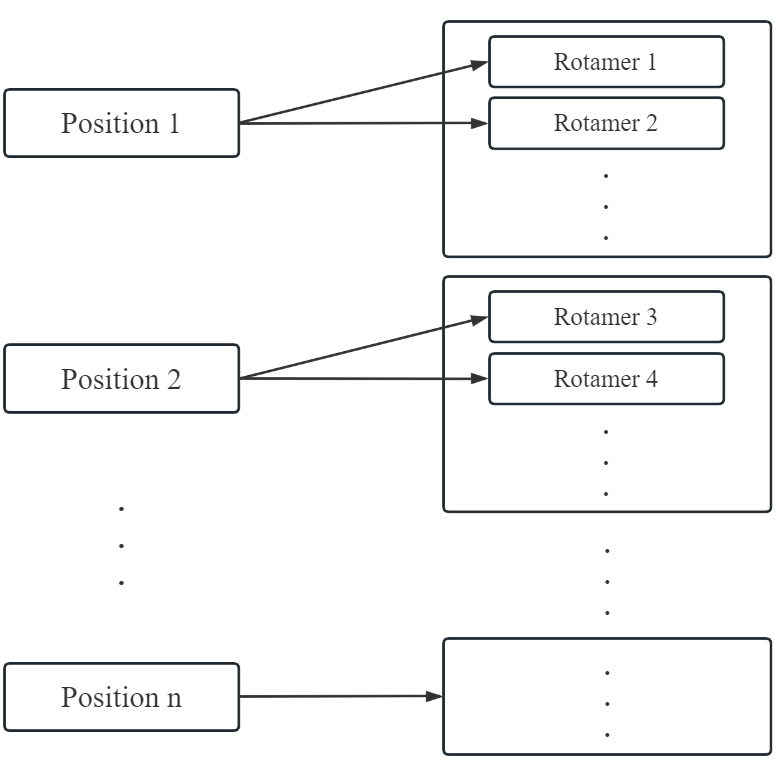}
\caption{The relationship between positions and rotamers in the CPD problem}\label{p_r}
\end{figure}

Briefly, the CPD problem is to select a rotamer among a set of them for each position on the backbone to minimize the total energy, as shown in \textbf{Figure} \ref{p_r}. To be more specific, the CPD problem could be described as the following mathematical model.

Let $n$ be the number of positions on the backbone of the protein, and we use $[n]$ to denote the set of integers $\left\{1,2,...,n\right\}$. $I^{(i)}$ is a set of rotamers that one can choose for position $i\in [n]$, where $i_r$ corresponds to the $r$-th rotamer that can be chosen for position $i$. Let $l_i=\vert I^{(i)}\vert$, that is, the number of elements in the set $I^{(i)}$. Let $m=\sum\limits^n_{i=1}l_i$. Define $x\in\mathbb{R}^m$ as follows:

\begin{equation}\label{vectorx}
x=
\begin{pmatrix}
x^{(1)}\\ \vdots\\x^{(n)}
\end{pmatrix}
=
\begin{pmatrix}
x^{(1)}_1\\ \vdots \\x^{(1)}_{l_1}\\ \vdots\\x^{(n)}_1\\ \vdots \\x^{(n)}_{l_n}
\end{pmatrix}
\in\mathbb{R}^m.\notag
\end{equation}
Here $x^{(i)}\in\mathbb{R}^{l_i}$ is the $i$-th block of the assignment variable $x$, $i\in [n]$, which satisfies
\begin{equation}
x^{(i)}_r=
\begin{cases}
1,\enspace\hbox{if the }r\hbox{-th rotamer is assigned to position }i,\enspace i\in [n],\\
0,\enspace\hbox{otherwise.}\notag
\end{cases}
\end{equation}
The CPD problem is to pick up one specific rotamer for each position of the protein, such that the total energy of the protein is minimized, that is,

\begin{equation}\label{CPD}
\begin{split}
\min\limits_d&\enspace\sum\limits_{i,r}E(i_r)\cdot x^{(i)}_r+\sum\limits_{i,r,j,s,i<j}E(i_r,j_s)\cdot x^{(i)}_r\cdot x^{(j)}_s\\
\hbox{s.t.}&\enspace\sum\limits_{r\in I^{(i)}} x^{(i)}_r=1,\enspace\forall i\in [n],\\
&\enspace x^{(i)}_r\in\left\{0,1\right\},\enspace\forall r\in I^{(i)},\enspace\forall i\in [n],
\end{split}
\end{equation}
where $E(i_r)$ represents the energy contribution of rotamer $r$ at position $i$ capturing internal interactions (and a reference energy for the associated amino acid) or interactions with fixed regions, and $E(i_r,j_s)$ represents the pairwise interaction energy between rotamer $r$ at position $i$ and rotamer $s$ at position $j$ \cite{desmet1992dead}. Both $E(i_r)$ and $E(i_r,j_s)$ are constant and already known for any $i$, $j\in [n]$, $r\in I^{(i)}$, $s\in I^{(j)}$. 


Problem (\ref{CPD}) is a quadratic semi-assignment problem (QSAP), and it can be equivalently expressed in the matrix form as follows.
Define $a\in\mathbb{R}^m$ and $B\in\mathbb{R}^{m\times m}$ as
\begin{equation}\label{vectora}
a=
\begin{pmatrix}
a^{(1)}\\ \vdots\\a^{(n)}
\end{pmatrix}
=
\begin{pmatrix}
a^{(1)}_1\\ \vdots \\a^{(1)}_{l_1}\\ \vdots\\a^{(n)}_1\\ \vdots \\a^{(n)}_{l_n}
\end{pmatrix}
\in\mathbb{R}^m,\enspace B=
\begin{pmatrix}
\bf{0}_{l_1\times l_1} & B_{12} & \cdots & B_{1n} \\
B_{12}^\top  & \bf{0}_{l_2\times l_2} & \cdots & B_{2n} \\
\vdots & \vdots & \ddots & \vdots \\
B_{1n}^\top  & B_{2n}^\top  & \cdots & \bf{0}_{l_n\times l_n} \\
\end{pmatrix}
\in\mathbb{R}^{m\times m},
\end{equation}
where
$a^{(i)}_r=E(i_r)$, $r\in [l_i]$
, $i\in [n]$, and
\begin{equation}
B_{ij}=
\begin{pmatrix}
b^{ij}_{11} & b^{ij}_{12} & \cdots & b^{ij}_{1l_j}\\
b^{ij}_{21} & b^{ij}_{22} & \cdots & b^{ij}_{2l_j}\\
\vdots & \vdots & \ddots & \vdots\\
b^{ij}_{l_i1} & b^{ij}_{l_i2} & \cdots & b^{ij}_{l_il_j}\\
\end{pmatrix}
\in\mathbb{R}^{l_i\times l_j},\enspace b^{ij}_{rs}=E(i_r,j_s),\enspace i<j,\enspace i,\enspace j\in [n].\notag
\end{equation}

Using the above notations, the objective function of the CPD problem can be represented by
\begin{equation}\label{objective}
f(x)=\frac{1}{2}x^\top Bx+a^\top x,
\end{equation}
and therefore the CPD problem (\ref{CPD}) can be equivalently expressed as the following QSAP
\begin{equation}\label{assignment}
\begin{split}
\mathop{\min}_{x\in\mathbb{R}^m}\enspace& f(x)\\
\hbox{s.t.}\enspace&\sum\limits_{r\in [l_i]}x^{(i)}_r=1,\enspace i\in [n],\\
&x^{(i)}_r\in \left\{0,1\right\},\enspace r\in [l_i],\enspace i\in [n].
\end{split}
\end{equation}

\section{An Equivalent Continuous Relaxation Problem}\label{sec3}

In this part, we show the equivalence between the CPD problem (\ref{assignment}) and its continuous relaxation problem. 

Like many other quadratic assignment problems such as the traveling salesman problem \cite{gavish1978travelling}, the bin-packing problem \cite{martello1990bin} and the max clique problem \cite{bomze1999maximum}, the CPD problem (\ref{assignment}) is also NP-hard \cite{pierce2002protein,traore2013new} in general. Therefore, there are no polynomial time algorithm to solve problem (\ref{assignment}), unless P=NP. In the next section, we will propose an efficient algorithm to obtain a high-quality solution of problem (\ref{assignment}) based on its continuous relaxation. 

A natural way to solve problem (\ref{assignment}) is to consider its relaxation problem as follows:
\begin{equation}\label{relaxation}
\begin{split}
\mathop{\min}_{x\in\mathbb{R}^{m}}\enspace&f(x)\\
\hbox{s.t.}\enspace&\sum\limits_{r\in [l_i]}x^{(i)}_r=1,\enspace i\in [n],\\
&x^{(i)}_r\geqslant 0,\enspace r\in [l_i],\enspace i\in [n].
\end{split}
\end{equation}

After relaxing $x^{(i)}_r\in \left\{0,1\right\}$ to $x^{(i)}_r\in [0,1]$, the feasible region in (\ref{relaxation}) becomes larger than that of (\ref{assignment}). Therefore, a natural question is, what is the relationship between the gloabl minimizer of (\ref{assignment}) and the global minimizer of (\ref{relaxation})? To answer it, we first show the following proposition. 

\begin{proposition}\label{proposition1}
$f(x)$ is a linear function with respect to each block $x^{(i)}$, $i\in [n]$, where $f(x)$ is defined as in (\ref{objective}), and $a$ and $B$ are defined as in (\ref{vectora}).
\end{proposition}

\begin{proof}
In fact, $\nabla f(x)=a+Bx$, and $\nabla_{x^{(i)}}f(x)$ takes the following form.
\begin{equation}\label{derivative}
\nabla_{x^{(i)}}f(x)=a^{(i)}+\sum\limits_{j\neq i}B_{ij}x^{(j)}=
\begin{pmatrix}
a^{(i)}_1\\a^{(i)}_2\\ \vdots\\a^{(i)}_{l_i}
\end{pmatrix}
+
\begin{pmatrix}
\sum\limits_{j\neq i}\sum\limits_{s\in [l_j]}b^{ij}_{1s}x^{(j)}_s\\
\sum\limits_{j\neq i}\sum\limits_{s\in [l_j]}b^{ij}_{2s}x^{(j)}_s\\
\cdot \\ \cdot \\ \cdot \\
\sum\limits_{j\neq i}\sum\limits_{s\in [l_j]}b^{ij}_{l_is}x^{(j)}_s\\
\end{pmatrix}
,\enspace i\in [n].
\end{equation}
As shown above, for every $i\in [n]$, the $\nabla_{x^{(i)}}f(x)$ is determined by all the blocks $x^{(j)}$, $j\neq i$, and is not related to the block $x^{(i)}$. In other words, $f(x)$ is a linear function with respect to each block $x^{(i)}$, $i\in [n]$.
\end{proof}

Due to this linear property of $f(x)$ with respect to each block $x^{(i)}$, $i\in [n]$, $f(x)$ can be written as a function of $x^{(i)}$ and $x^{(-i)}$, where
\begin{equation}\label{variable}
x^{(-i)}=
(x^{(1)},...,x^{(i-1)},x^{(i+1)},...,x^{(n)})^\top
\in\mathbb{R}^{m-l_i},\enspace i\in [n].
\end{equation}
More specifically,
\begin{equation}\label{function}
f(x)=\nabla_{x^{(i)}}f(x)^\top x^{(i)}+f_{-i}(x^{(-i)}),\enspace i\in [n].
\end{equation}
Here $\nabla_{x^{(i)}}f(x)$ is only related to $x^{(-i)}$, and $f_{-i}(x^{(-i)})$ represents the part in $f(\cdot)$ which is only related to $x^{(-i)}$. By such particular structure of the objective function, we have the following theorem, which shows that the continuous relaxation problem (\ref{relaxation}) is actually equivalent to the original CPD problem (\ref{assignment}).\\

\begin{theorem}\label{theorem1}
Let $\Vert x\Vert_0$ denote the number of nonzero elements in $x$.
There exists an optimal solution 
\begin{math}
x^*
\end{math}
to the relaxation problem (\ref{relaxation}) such that 
\begin{math}
\Vert x^*\Vert_0=n.
\end{math}
Moreover,
\begin{math}
x^*
\end{math}
is also an optimal solution to the original problem {\rm (\ref{assignment})}.
\end{theorem}

\begin{proof}
We proceed with the proof by showing that for problem (\ref{relaxation}), there exists a global optimal solution $x^*$ such that each block ${x^*}^{(i)}$ is an extreme point of the simplex set $\Delta_i$ defined by $\Delta_i=\left\{y\in\mathbb{R}^{l_i}\text{ }|\text{ } y^\top e_i=1,y\geqslant0\right\}$, $i\in [n]$, where $e_i\in\mathbb{R}^{l_i}$ is a vector with all entries equal to one. Let $\overline{x}$ be a global optimal solution of problem (\ref{relaxation}). Suppose that there exists one block as $\overline{x}^{(j)}$, such that $\overline{x}^{(j)}$ is not an extreme point of $\Delta_j$, i.e., $\Vert \overline{x}^{(j)}\Vert_0>1$. Clearly, point $\overline{x}^{(j)}$ is an optimal solution of the linear programming problem with simplex constraints
\begin{equation}\label{simplex}
\min\limits_{y\in\Delta_j}\nabla_{x^{(j)}}f(\overline{x})^\top y+f_{-j}(\overline{x}^{(-j)}),
\end{equation}
where $\overline{x}^{(-j)}$ is defined similarly as in (\ref{variable}). From the basic linear programming theory, there must exist an extreme point $\overline{y}\in\Delta_j$ such that $\overline{y}$ is an optimal solution of problem (\ref{simplex}). Then we must have $\nabla_{x^{(j)}}{f(\overline{x})}^\top \overline{y}+f_{-j}(\overline{x}^{(-j)})=f(\overline{x})$. Define a new point $\hat{x}\in\mathbb{R}^m$ by
\begin{equation}
\hat{x}^{(i)}=
\begin{cases}
\overline{y},\text{ if }i=j,\\
\overline{x}^{(i)},\text{ if }i\neq j,\text{ }i\in [n].
\end{cases}
\end{equation}
We have $f(\hat{x})=f(\overline{x})$ and hence $\hat{x}$ is a global minimizer of problem (\ref{relaxation}). If each block in $\hat{x}$, i.e., $\hat{x}^{(i)}$, $i\in [n]$, is an extreme point of the set $\Delta_i$, let $x^*=\hat{x}$. The proof is finished. Otherwise, repeat the above process. After at most $k$ steps $(k\leqslant n)$, one will reach a global minimizer $x^*$, such that each block in $x^*$, namely ${x^*}^{(i)}$, $i\in [n]$, is an extreme point of the set $\Delta_i$. This completes the proof.
\end{proof}

\begin{remark}
The proof of \textbf{Theorem} \ref{theorem1} makes use of the properties of linear programming. Another way to prove the result is to apply Corollary 2 in \cite{hypergraph} as well as \textbf{Proposition} \ref{proposition1}.\\
\end{remark}

\textbf{Theorem} \ref{theorem1} basically reveals that problem (\ref{relaxation}) is a tight continuous relaxation of problem (\ref{assignment}) in the sense that two problems share at least one global minimizer. 

Given a global minimizer of the relaxation problem (\ref{relaxation}), from the proof of \text{Theorem} \ref{theorem1}, we can use the following algorithm to get the global minimizer of the original problem (\ref{assignment}). In this way, we can obtain the solutions of the original CPD problems by just solving their relaxations.

\begin{algorithm}\label{algorithm1}
\caption{Obtain the global minimizer of (\ref{assignment}) by that of (\ref{relaxation})}\label{algorithm1}
\begin{algorithmic}[1]
\Require a global optimal solution $x^0\in\mathbb{R}^m$ of the relaxation problem (\ref{relaxation}).
\Ensure a global optimal solution $\hat{x}\in\mathbb{R}^m$ of the original problem (\ref{assignment}).
\State initialization: $l=0$.
\While{$\Vert x^l\Vert_0>n$}

		for $i=1,...,n$, find the first block of $x^l$, denoted as ${(x^l)}^{(j)}$, such that $\Vert {(x^l)}^{(j)}\Vert_0>1$, choose one index $s^0$ from $\Gamma_j(x^l):=\left\{s:{(x^l)}^{(j)}_s>0\right\}$, and define $x^{l+1}$ as:

		\begin{math}
		(x^{l+1})^{(j)}_s=
		\begin{cases}
		1, &s=s^0,\\
		0, &\hbox{otherwise,}
		\end{cases}
		\enspace {(x^{l+1})}^{(i)}=
		\begin{cases}
		{(x^{l+1})}^{(j)}, &i=j,\\
		{(x^l)}^{(i)}, &\hbox{otherwise.}
		\end{cases}
		\end{math}
        
		$l=l+1$.
\EndWhile\\
output $\hat{x}=x^l.$
\end{algorithmic}
\end{algorithm}

\begin{remark}
For any approximate solution of the relaxation problem (\ref{relaxation}), we can also use \textbf{Algorithm} \ref{algorithm1} to obtain a feasible solution of the original problem (\ref{assignment}). 
\end{remark}

\section{The Numerical Algorithm for Problem (\ref{relaxation})}\label{sec4}

Next we design an efficient algorithm for the relaxation problem (\ref{relaxation}). It should be noticed that the aim we solve (\ref{relaxation}) is to identify the locations of nonzero entries of the global minimizer of (\ref{relaxation}), rather than to find the magnitude of it. This is because once the locations of the nonzero entries are identified, we can apply \textbf{Algorithm} \ref{algorithm1} to obtain a global optimal solution of (\ref{assignment}). From this point of view, keeping the equality constraints in (\ref{relaxation}) may not be necessary. Therefore, we apply the quadratic penalty method to solve (\ref{relaxation}). That is, we penalize the equality constraints to the objective function, and solve the following quadratic penalty subproblem
\begin{equation}\label{penalty}
\min\limits_{x\in\mathbb{R}^m}f(x)+\frac{\sigma}{2}\sum\limits_{i=1}^n\left(\sum\limits_{r\in [l_i]}x^{(i)}_r-1\right)^2\enspace\hbox{s.t. }x\geqslant0,
\end{equation}
where $\sigma>0$ is a penalty parameter. Compared with the relaxation problem (\ref{relaxation}), the penalty subproblem (\ref{penalty}) is much easier to solve since the constraints are greatly simplified again. Up to now, we have transformed the original discrete model (\ref{assignment}) into an equivalent continuous penalty model (\ref{penalty}), which is more amenable to the algorithmic design.

Based on the above, we demonstrate the quadratic algorithm in \textbf{Algorithm} \ref{algorithm2}, which obtains the optimal solution of the relaxation problem (\ref{relaxation}) numerically by solving the penalty subproblem (\ref{penalty}), and then transforms it into the optimal solution of the original problem (\ref{assignment}). It should be emphasized that during the process of solving the subproblem (\ref{penalty}), we only care about identifying the locations of nonzero entries of the global minimizer, which is also a significant part of the projected Newton method \cite{bertsekas1982projected}. As a consequence, we use the projected Newton method to solve the subproblem (\ref{penalty}).

\begin{algorithm}
\caption{Quadratic Penalty Method for (\ref{assignment})}\label{algorithm2}
\begin{algorithmic}[1]
\Require $x^0\geqslant0$, $\sigma_0>0$, $\rho>1$.
\Ensure an optimal solution $\hat{x}\in\mathbb{R}^m$ to the original problem (\ref{assignment}).
\State initialization: choose $x^0\in\mathbb{R}^m_+$, $k:=0$.
\While{the termination condition is not met}

		start from $x^k$, solve problem (\ref{penalty}) with $\sigma := \sigma_k$ to get $x^{k+1}$.

		update $\sigma_{k+1}:=\rho\sigma_k$, $k:=k+1$.
\EndWhile\\
transform $x^k$ into $\hat{x}$ by \textbf{Algorithm} 1.\\
output $\hat{x}$.
\end{algorithmic}
\end{algorithm}

\begin{remark}
In Step 2, we always start from $x^k$ instead of $x^0$, so that we can make better use of the known information and reduce the computational cost. Therefore, \textbf{Algorithm} \ref{algorithm2} enjoys a low per-iteration complexity.\\
\end{remark}

The following theorem addresses the convergence of the quadratic penalty method, which can be found in classic optimization books such as \cite{jorge2006numerical} (Theorem 17.1) and \cite{sun2006optimization} (Corollary 10.2.6). Therefore, the proof is omitted.

\begin{theorem}\label{theorem2}
Let $\left\{x^k\right\}$ be generated by \textbf{Algorithm} \ref{algorithm2}, and $\lim_{k\to\infty}\sigma_k=+\infty$. If each $x^{k+1}$ is a global minimizer of {\rm (\ref{penalty})}, then any accumulation point of the generated sequence $\left\{x^k\right\}$ is a global optimal solution of the relaxation problem {\rm (\ref{relaxation})}.
\end{theorem}

Due to \textbf{Theorem} \ref{theorem2}, we always assume the following holds.

\begin{assumption}\label{assumption1}
Let $\left\{x^k\right\}$ be generated by \textbf{Algorithm} \ref{algorithm2}, and $\lim_{k\to\infty}\sigma_k=+\infty$. Denote $K$ as a subset of $\left\{1,2,...\right\}$. Assume that $\lim_{k\to\infty,k\in K}x^k=z$, and $z$ is a global optimal solution of the relaxation problem {\rm (\ref{relaxation})}.
\end{assumption}

The following theorems further analyze the convergence of \textbf{Algorithm} \ref{algorithm2}. We define
\begin{equation}
\begin{split}
&\Gamma^k=\left\{p:{(x^k)}_p>0\right\},\enspace\Gamma(z)=\left\{p:z_p>0\right\},\\
&\Omega^k=\left\{p:{(x^k)}_p=0\right\},\enspace\Omega(z)=\left\{p:z_p=0\right\},\\
&J_i^k=\arg\max\limits_r\left\{(x^k)^{(i)}_r\right\},\enspace J_i(z)=\arg\max\limits_r\left\{z^{(i)}_r\right\},\\
&J^k=\left\{J_i^k,i\in [n]\right\},\enspace J(z)=\left\{J_i(z),i\in [n]\right\}.\notag
\end{split}
\end{equation}

\begin{theorem}\label{theorem3}
Suppose that \textbf{Assumption} \ref{assumption1} holds.
\begin{itemize}
\item [(i)]
If
$\Vert z\Vert_0=n$, then there exists an integer $k_0>0$, such that $J^k=\Gamma(z)$, $\forall k\geqslant k_0$, $k\in K$; 
\item [(ii)]
If 
\begin{math}
\Vert z\Vert_0>n$, and $\vert J_i(z)\vert=1$, $\forall i\in [n]$, then there exist an integer $
k_0>0$ and an optimal solution $x^*$ to the original problem {\rm (\ref{assignment})},  
such that $J^k=\Gamma^*$, $\forall k\geqslant k_0$, $k\in K;
\end{math}
\item [(iii)]
If
\begin{math}\Vert z\Vert_0>n$, and $\vert J_i(z)\vert>1$ for at least one $i\in [n]$, then there exist a subsequence $\left\{x^k\right\}$, $k\in K'\subseteq K$, an integer $k_0>0$, and an optimal solution $x^*$ to the original problem {\rm (\ref{assignment})}, such that $J^k=\Gamma^*$, $\forall k\geqslant k_0$, $k\in K'
\end{math}. 
\end{itemize}
\end{theorem}

\begin{proof}
see the appendix.
\end{proof}

\textbf{Theorem} \ref{theorem3} ensures that there always exists a subsequence of $\left\{x^k\right\}$ generated by \textbf{Algorithm} \ref{algorithm2} whose support set will coincide with the support set of one global minimizer of (\ref{assignment}). 

\begin{theorem}\label{theorem4}
Suppose that \textbf{Assumption} \ref{assumption1} holds. If there exists a positive integer 
\begin{math}
k_0,
\end{math}
such that $\Vert x^k\Vert_0=n$, $\forall k\geqslant k_0$, $k\in K$, then there is a positive integer $k_1\geqslant k_0$ such that $\Gamma^k=\Gamma(z)$, $\forall k\geqslant k_1$, 
$k\in K$, where z is an optimal solution of {\rm (\ref{assignment})}. 
\end{theorem}

\begin{proof}
see the appendix.
\end{proof}

\textbf{Theorem} \ref{theorem4} gives a special case when \textbf{Algorithm} \ref{algorithm2} converges, which indicates that we do not need to drive $\sigma_k$ to infinity since only the support set of $z$ is needed. In practice, if the conditions in \textbf{Theorem} \ref{theorem4} holds, we can stop the algorithm when the elements in $J^k$ keep unchanged for several iterations. Consequently, the above theorems provide a method to design the termination rule for \textbf{Algorithm} 2.

\section{Numerical Results}\label{sec5}

The proposed \textbf{Algorithm} \ref{algorithm2} is termed as AQPPG, which is the abbreviation of Assignment Quadratic Penalty Projected Gradient method.
We implement the algorithm in MATLAB (R2018a). All the experiments are performed on a Lenovo desktop with AMD Ryzen7 4800H CPU at 2.90 GHz and 16 GB of memory running Windows 10. We use the data as in \cite{allouche2014computational}, which can be downloaded from \url{https://genoweb.toulouse.inra.fr/~tschiex/CPD-AIJ/}.
\footnote{To convert the floating point energies of a given instance to non-negative integer costs, David Allouche et al. \cite{allouche2014computational} subtracted the minimum energy to all energies and then multiplied energies by an integer constant $M$ and rounded to the nearest integer. Therefore, all the energies in the data sets are non-negative integers.}
\textbf{Table} \ref{parameter} shows the information of all data sets tested. In \textbf{Table} \ref{parameter}, $Position$ $(n)$ represents the number of positions in the target protein, which is also the number of blocks in the decision variable $x$. 
$Component$ $(m)$ shows the dimension of the decision variable $x$.

\begin{longtable}[H]{|c|c|c|c|c|c|}
\caption{Information of all the data sets. $\min l_i$ represents how many components one block in $x$ contains at least, and $\max l_i$ represents how many components one block contains at most.}\label{parameter}
\endfirsthead
\multicolumn{5}{c}{{Table \ref{parameter} - continued from previous page}}\\ \hline
NO. & Data & {\makecell[c]{$n$ \\ Position}} & {\makecell[c]{$l_i\in (\min l_i, \max l_i)$ \\ Rotamer}} & {\makecell[c]{$m$ \\ Component}}\\\hline
\endhead
\hline\multicolumn{3}{l}{{Continued on next page}}\\
\endfoot
\hline
\endlastfoot
\hline
NO. & Data & {\makecell[c]{$n$ \\ Position}} & {\makecell[c]{$l_i\in (\min l_i, \max l_i)$ \\ Rotamer}} & {\makecell[c]{$m$ \\ Component}}\\\hline
1 & 1HZ5 & 12 & (49, 49) & 588 \\ 
2 & 1PGB & 11 & (49, 49) & 539 \\ 
3 & 2PCY & 18 & (48, 48) & 864 \\ 
4 & 1CSK & 30 & (3, 49) & 616 \\ 
5 & 1CTF & 39 & (3, 56) & 1204 \\ 
6 & 1FNA & 38 & (3, 48) & 990 \\ 
7 & 1PGB & 11 & (198, 198) & 2178 \\ 
8 & 1UBI & 13 & (49, 49) & 637 \\ 
9 & 2TRX & 11 & (48, 48) & 528 \\ 
10 & 1UBI & 13 & (198, 198) & 2574 \\ 
11 & 2DHC & 14 & (198, 198) & 2772 \\ 
12 & 1PIN & 28 & (198, 198) & 5544 \\ 
13 & 1C9O & 55 & (198, 198) & 10890 \\ 
14 & 1C9O & 43 & (3, 182) & 1950 \\ 
15 & 1CSE & 97 & (3, 183) & 1355 \\ 
16 & 1CSP & 30 & (3, 182) & 1114 \\ 
17 & 1DKT & 46 & (3, 190) & 2243 \\ 
18 & 1BK2 & 24 & (3, 182) & 1294 \\ 
19 & 1BRS & 44 & (3, 194) & 3741 \\ 
20 & 1CM1 & 17 & (198, 198) & 3366 \\ 
21 & 1SHG & 28 & (3, 182) & 737 \\ 
22 & 1MJC & 28 & (3, 182) & 493 \\ 
23 & 1SHF & 30 & (3, 56) & 638 \\ 
24 & 1FYN & 23 & (3, 186) & 2474 \\ 
25 & 1NXB & 34 & (3, 56) & 800 \\ 
26 & 1TEN & 39 & (3, 66) & 808 \\ 
27 & 1POH & 46 & (3, 182) & 943 \\ 
28 & 1CDL & 40 & (3, 186) & 4141 \\ 
29 & 1HZ5 & 12 & (198, 198) & 2376 \\ 
30 & 2DRI & 37 & (3, 186) & 2120 \\ 
31 & 2PCY & 46 & (3, 56) & 1057 \\ 
32 & 2TRX & 61 & (3, 186) & 1589 \\ 
33 & 1CM1 & 42 & (3, 186) & 3633 \\ 
34 & 1LZ1 & 59 & (3, 57) & 1467 \\ 
35 & 1GVP & 52 & (3, 182) & 3826 \\ 
36 & 1R1S & 56 & (3, 182) & 3276 \\ 
37 & 2RN2 & 69 & (3, 66) & 1667 \\ 
38 & 1HNG & 85 & (3, 182) & 2341 \\ 
39 & 3CHY & 74 & (3, 66) & 2010 \\ 
40 & 1L63 & 83 & (3, 182) & 2392 \\ 
\end{longtable}

\subsection{An example as illustration}

We first demonstrate the performance of AQPPG on the data set 1CSK. The target protein in 1CSK contains 30 positions, which means that the decision variables have 30 blocks, i.e., $n=30$. Each position contains at least 3, and at most 49 kinds of rotamers to be selected, which means that each block in the decision variable has 3 to 49 components, i.e., $3\leqslant l_i\leqslant 49$, $i\in [n]$. The numbers of rotamers that can be selected for each position are shown in \textbf{Figure} \ref{l_i}. The decision variable is a 616-dimensional vector.

\begin{figure}[H]
\centering
\includegraphics[width=.85\linewidth]{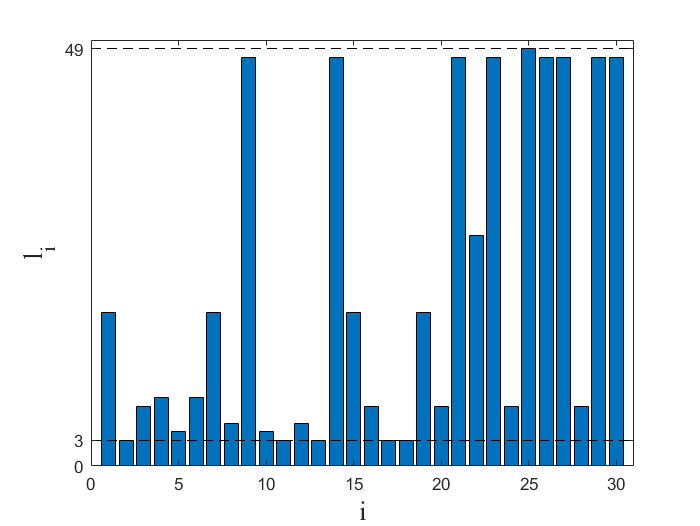}
\caption{The number of rotamers $l_i$ in each position $i$ for 1CSK}\label{l_i}
\end{figure}

\begin{table}[!ht]
\centering
\caption{Selected rotamers for each position $i$ in $\hat{x}$ for 1SHF}\label{selected}
\begin{tabular}{|c|c|c|c|c|c|c|c|c|c|c|} \hline
Position $i$ & 1 & 2 & 3 & 4 & 5 & 6 & 7 & 8 & 9 & 10 \\ \hline
Selected rotamer & 7 & 3 & 6 & 5 & 2 & 5 & 16 & 5 & 18 & 1 \\ \hline
Position $i$ & 11 & 12 & 13 & 14 & 15 & 16 & 17 & 18 & 19 & 20\\ \hline
Selected rotamer & 3 & 5 & 2 & 23 & 11 & 6 & 3 & 2 & 1 & 7 \\ \hline
Position $i$ & 21 & 22 & 23 & 24 & 25 & 26 & 27 & 28 & 29 & 30 \\ \hline
Selected rotamer & 30 & 12 & 12 & 2 & 2 & 18 & 2 & 6 & 12 & 3 \\ \hline
\end{tabular}
\end{table}

\noindent The AQPPG reaches the termination condition after 1 second (459 iterations). According to the optimal strategy given by the algorithm, the total protein energy, i.e., the optimization goal, reaches the minimum value 1125838 when specific rotamers are selected for the corresponding positions, as shown in \textbf{Table} \ref{selected}. The following \textbf{Figure} \ref{spar}, \textbf{Figure} \ref{sparnum} and \textbf{Figure} \ref{value} show more details during the iteration process. 

\begin{figure}[htbp]
\centering
\begin{minipage}{0.49\linewidth}
\centering
\includegraphics[scale=0.28]{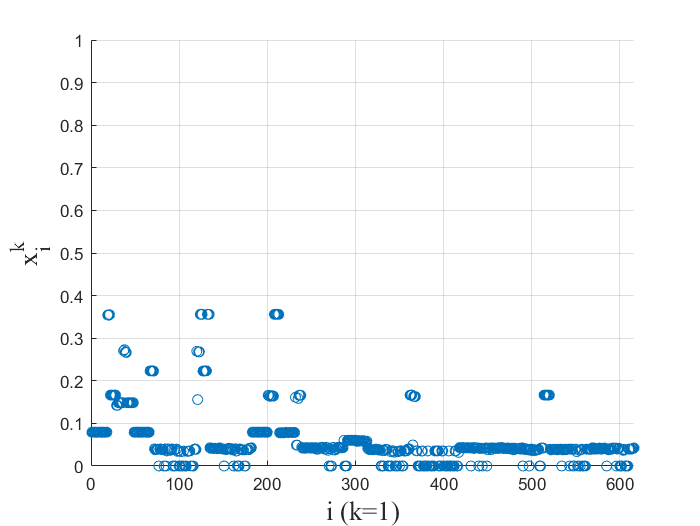}
\end{minipage}
\begin{minipage}{0.49\linewidth}
\centering
\includegraphics[scale=0.28]{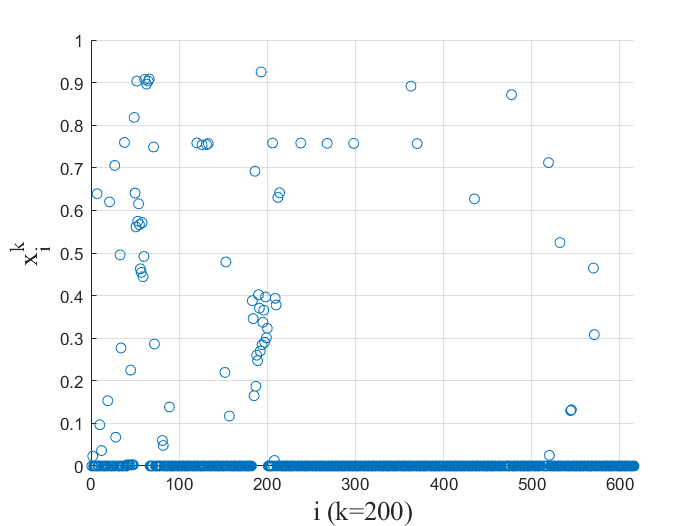}
\end{minipage}

\begin{minipage}[b]{0.49\linewidth}
\centering
\includegraphics[scale=0.28]{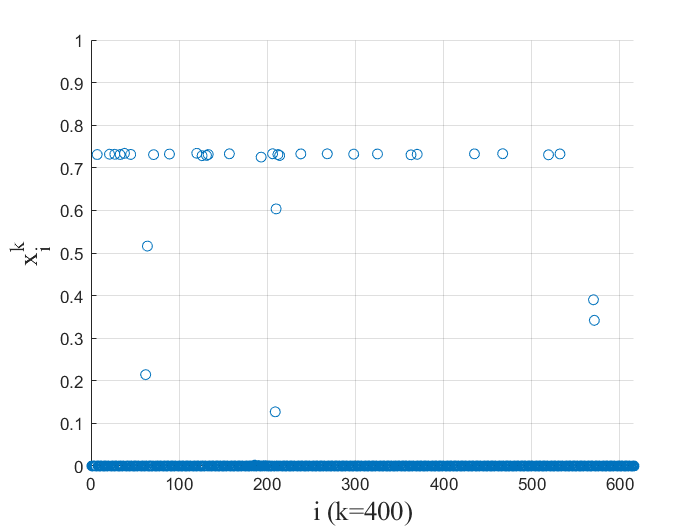}
\end{minipage}
\begin{minipage}[b]{0.49\linewidth}
\centering
\includegraphics[scale=0.28]{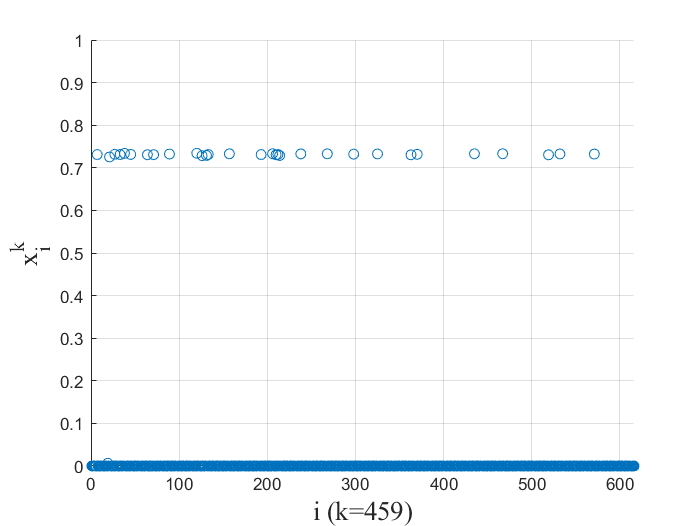}
\end{minipage}
\caption{Nonzeros of $x^k$ when $k=1$, $200$, $400$ and $459$ for 1CSK}\label{spar}
\end{figure}


\begin{figure}[H]
\centering
\includegraphics[width=.75\linewidth]{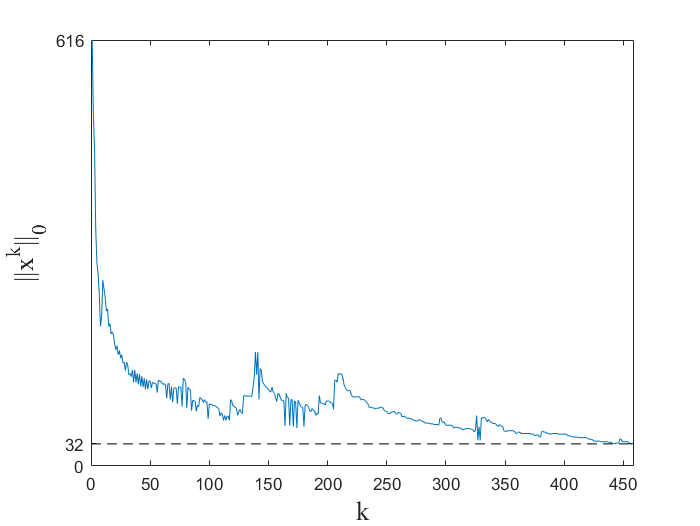}
\caption{The number of nonzero components in $x^k$ for 1CSK}\label{sparnum}
\end{figure}

\begin{figure}[h]
\centering
\includegraphics[width=.75\linewidth]{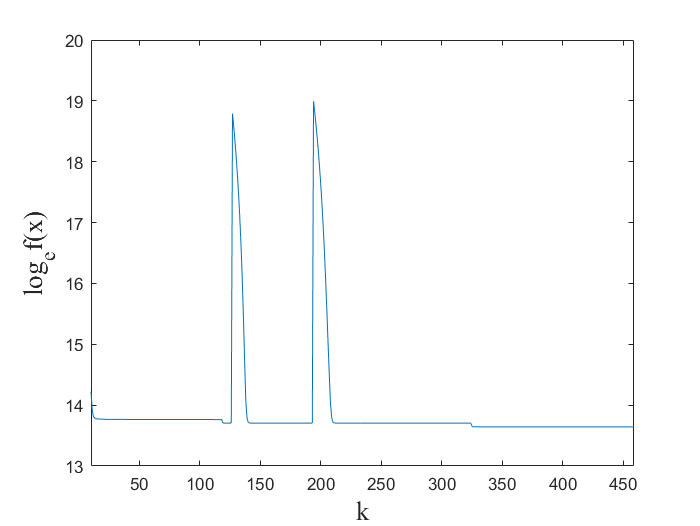}
\caption{Function value during the iterations for 1CSK}\label{value}
\end{figure}

\noindent As shown in \textbf{Figure} \ref{spar} and \textbf{Figure} \ref{sparnum}, the number of nonzero components in the decision variable $x^k$ decreased gradually during the iteration process, and the number of nonzero entries finally turned into $32$, which means that the decision variable $x^k$ was close to the feasible region of the problem (\ref{assignment}). 
\textbf{Figure} \ref{value} shows the function value during the iterations. In the first 10 iterations, the function value dropped dramatically from the initial value, which was over $5\times 10^8$. When $k$ was about 130 and 200, there were several fluctuations to the function value, meaning that the alogorithm was searching for better staionary points. After that, the function value decreased rapidly and gradually stabilized at 629244, which was the optimal value of the relaxation problem (\ref{relaxation}). However, note that the decision variable $x^k$ is not a feasible point for the original problem (\ref{assignment}), we need to transform $x^k$ into $\hat{x}$ using \textbf{Algorithm} \ref{algorithm1}. Thus we get the solution $\hat{x}$ as shown in \textbf{Table} \ref{selected}, and the corresponding optimal value for the problem (\ref{assignment}) is 1125838.

\subsection{Comparison with the state-of-the-art branch-and-cut solver}

We compare AQPPG with Gurobi (version 9.5.2), one of the state-of-the-art branch-and-cut solver. 
\textbf{Table} \ref{results} shows the results given by AQPPG and Gurobi. 
$Objective$ represents the function values given by different methods. Results marked with * means that the corresponding solver does not terminate within 10 hours, and the objective is the best value the solver could give within 10 hours. '$-$' means that the solver fail to solve the problem for the lack of memory. 
$Ratio$ represents the ratio of the optimal values given by Gurobi compared to those given by AQPPG. $Time$ shows the CPU time to get the optimal values in the form of $hh:mm:ss$.

\begin{longtable}[H]{|c|c|rr|r|rr|}
\caption{Results given by AQPPG and Gurobi}\label{results}
\endfirsthead
\multicolumn{6}{c}{{Table \ref{results} - continued from previous page}}\\ \hline
\multirow {2} {0.75cm} {NO.} & \multirow {2} {0.75cm} {Data} & \multicolumn {2} {c|} {Objective} & \multicolumn {1} {c|} {Ratio} & \multicolumn {2} {c|} {Time} \\
&& \multicolumn {1} {r} {AQPPG} & \multicolumn {1} {r|} {Gurobi} & \multicolumn {1} {c|} {Gurobi} & \multicolumn {1} {r} {AQPPG} & \multicolumn {1} {r|} {Gurobi} \\ \hline
\endhead
\hline\multicolumn{3}{l}{{Continued on next page}}\\
\endfoot
\hline
\endlastfoot
\hline
\multirow {2} {0.75cm} {NO.} & \multirow {2} {0.75cm} {Data} & \multicolumn {2} {c|} {Objective} & \multicolumn {1} {c|} {Ratio} & \multicolumn {2} {c|} {Time} \\
&& \multicolumn {1} {r} {AQPPG} & \multicolumn {1} {r|} {Gurobi} & \multicolumn {1} {c|} {Gurobi} & \multicolumn {1} {r} {AQPPG} & \multicolumn {1} {r|} {Gurobi} \\ \hline
1 & 1HZ5 & 150714 & 150714 & 100.00\% & 1 & 4:32:07 \\ 
2 & 1PGB & 125306 & 125306 & 100.00\% & 1 & 5:13:59 \\  
3 & 2PCY & 308545 & 307667 & 99.72\% & 2 & 9:58:45 \\  
4 & 1CSK & 1125838 & *1125798 & 100.00\% & 1 & 10:00:00 \\ 
5 & 1CTF & 1882883 & *1881874 & 99.95\% & 54 & 10:00:00 \\ 
6 & 1FNA & 3751671 & *3750260 & 99.96\% & 5 & 10:00:00 \\ 
7 & 1PGB & 287413 & *286468 & 99.67\% & 3:58 & 10:00:00 \\  
8 & 1UBI & 159700 & 159522 & 99.89\% & 2 & 5:32:53 \\ 
9 & 2TRX & 178900 & 178534 & 99.80\% & 1 & 4:34:26 \\  
10 & 1UBI & 382033 & *381180 & 99.78\% & 47 & 10:00:00 \\ 
11 & 2DHC & 1424025 & *1422718 & 99.91\% & 11 & 10:00:00 \\ 
12 & 1PIN & 1996834 & *1995099 & 99.91\% & 2:02 & 10:00:00 \\ 
13 & \underline{1C9O} & 8084802 & - & - & 3:40 & - \\ 
14 & 1C9O & 4975017 & *4959931 & 99.70\% & 1:19 & 10:00:00 \\ 
15 & 1CSE & 18602843 & *18602292 & 100.00\% & 41 & 10:00:00 \\ 
16 & 1CSP & 2521159 & *2520706 & 99.98\% & 28 & 10:00:00 \\ 
17 & 1DKT & 4214282 & *4192707 & 99.49\% & 8:13 & 10:00:00 \\ 
18 & 1BK2 & 1140948 & *1133737 & 99.37\% & 1:09 & 10:00:00 \\ 
19 & 1BRS & 4017422 & *4007755 & 99.76\% & 3:39 & 10:00:00 \\ 
20 & 1CM1 & 746221 & *743645 & 99.66\% & 56 & 10:00:00  \\ 
21 & 1SHG & 1513349 & 1513151 & 99.99\% & 3 & 5:27:03  \\ 
22 & \underline{1MJC} & 1514481 & - & - & 3 & -  \\ 
23 & 1SHF & 1101835 & 1101033 & 99.93\% & 1 & 7:18:47  \\ 
24 & 1FYN & 1194046 & *1183722 & 99.14\% & 1:08 & 10:00:00  \\ 
25 & 1NXB & 2979543 & *2971624 & 99.73\% & 3 & 10:00:00  \\ 
26 & 1TEN & 1960322 & *1959862 & 99.98\% & 5 & 10:00:00  \\ 
27 & 1POH & 4034259 & 4033915 & 99.99\% & 2:48 & 8:04:23  \\ 
28 & 1CDL & 3594181 & 3590578 & 99.90\% & 6:47 & 2:45:33  \\ 
29 & 1HZ5 & 343021 & *343113 & 100.03\% & 17 & 10:00:00  \\ 
30 & 2DRI & 2908142 & *2905276 & 99.90\% & 1:07:08 & 10:00:00  \\ 
31 & 2PCY & 2937638 & *2935820 & 99.94\% & 5 & 10:00:00  \\ 
32 & 2TRX & 7020438 & *7016199 & 99.93\% & 1:31 & 10:00:00  \\ 
33 & 1CM1 & 3904719 & *3895736 & 99.77\% & 2:40 & 10:00:00  \\ 
34 & 1LZ1 & 7038826 & *7022768 & 99.77\% & 1:37 & 10:00:00  \\ 
35 & 1GVP & 5205320 & *5196913 & 99.84\% & 1:33 & 10:00:00  \\ 
36 & 1R1S & 6174155 & *6171802 & 99.96\% & 10:34 & 10:00:00  \\ 
37 & 2RN2 & 8918311 & *8910166 & 99.91\% & 18 & 10:00:00  \\ 
38 & 1HNG & 13543984 & *13532638 & 99.91\% & 3:07 & 10:00:00  \\ 
39 & 3CHY & 10466158 & *10461537 & 99.96\% & 24 & 10:00:00  \\ 
40 & 1L63 & 13015089 & *12891316 & 99.05\% & 44 & 10:00:00  \\ 
\end{longtable}

\noindent We can see that our algorithm AQPPG could effectively solve CPD problems. Gaps between the solutions given by Gurobi and AQPPG range from -0.95\% to +0.03\%. However, compared with Gurobi, the proposed AQPPG is much more efficient. 
In most cases, AQPPG outperforms Gurobi by three order of magnitude in CPU time, and the CPU time for Gurobi to reach the optimal solution exceeds 10 hours in nearly all the cases. Specifically, Gurobi even fails to find feasible points in some certain cases such as 1C9O and 1MJC, while AQPPG could still terminate in a short period of time (3 minutes 40 seconds and 3 seconds).

Based on the above results, we can conclude that our proposed AQPPG can effectively find a high-quality solution within a reasonable amount of time.

\section{Conclusion}\label{sec6}

In this paper, we proposed an efficient algorithm called AQPPG for solving the CPD problem. Using the fact that the objective of CPD problem relies linearly on each block of the decision variable, we proved that any optimal solution to the relaxation problem (\ref{relaxation}) can be transformed into an optimal solution to the original problem (\ref{assignment}). Then we proposed AQPPG, a quadratic penalty method applied to solve the proposed relaxation problem. Our numerical results show that our proposed algorithm can effectively find a high-quality solution for the CPD problem, and is much more efficient than the state-of-the-art solver Gurobi.
\\

\section*{Data availability statement}
The datasets used in this paper can be downloaded from \url{https://genoweb.toulouse.inra.fr/~tschiex/CPD-AIJ/}.

\section*{Disclosure statement}
The authors have no conflict of interest to declare that are relevant to the content of this article.

\section*{Funding}
The work of Qing-Na Li is supported by the National Natural Science Foundation of China (NSFC) 12071032.

\begin{appendices}\label{appendix}

\section{Proofs of Theorem \ref{theorem3} and Theorem \ref{theorem4}}\label{appendixa}

\begin{proof}[Proof of Theorem~{\upshape\ref{theorem3}}]

(i) Together with \textbf{Theorem} \ref{theorem1} and 
\begin{math}
\Vert z\Vert_0=n,
\end{math}
$z$ must be a global minimizer of (\ref{assignment}). By the definition of $\Gamma(z)$ and $\Omega(z)$, there exists an integer $k_0>0$, such that for $k\geqslant k_0,\enspace k\in K$, there is 
$z_p>z_{p'},\enspace\forall p\in \Omega(z),\enspace p'\in \Gamma(z)$. This gives (i).

\noindent (ii)
\begin{math}
\Vert J_i(z)\Vert=1
\end{math}
implies that for $k$ sufficiently large, there is
\begin{equation}
(x_k)^{(i)}_{J_i(z)}>(x_k)^{(i)}_p,\enspace
\forall p\notin J_i(z),\enspace i=1,...,n.\notag
\end{equation}
Consequently, there is 
\begin{math}
J_i^k=J_i(z).
\end{math}
Now let $x_0=z$. Similar to the arguments in the proof of \textbf{Theorem} \ref{theorem1}, we construct $x_1$ by choosing
$p_0=J_i(z)$. Then we can obtain a finite sequence $x_0, x_1,...,x_r$ with
\begin{equation}
\Vert x_r\Vert_0<\cdot\cdot\cdot<\Vert x_1\Vert_0<\Vert x_0\Vert_0\notag
\end{equation}
After at most m steps, the process will stop. In other words, $1\leqslant r\leqslant m$. At the final point, $x_r$ 
will satisfy that $\Vert x_r\Vert_0=n$. One can find a global minimizer $x^*=x_r$ of (\ref{relaxation}) with sparsity $n$. Further, $x^*$ is a global minimizer of (\ref{assignment}) which satisfies
\begin{equation}
\Vert J_i^*\Vert_0=1,\enspace J_i^*=\Gamma_i^*=J_i(z)=J_i^k.\notag
\end{equation}
Consequently, (ii) holds.

\noindent (iii)
Suppose that there exists an index $q_1$ such that $\Vert J_{q_1}(z)\Vert>1$. Consequently, there exists $p_1\in J_{q_1}^k$, such that for $k\in K$ sufficiently large, there are infinite number of $k$ satisfying $J_{q_1}^k=p_1$. Denote the corresponding subsequence as $\left\{x_k\right\}_{k\in K_1}$, where $K_1\subset K$. Similarly, for $\Vert J_{q_2}(z)\Vert_0>1$, we can find an infinite number of $k\in K_2\subseteq K_1$ such that $J_{q_2}^k=p_2$. Repeating the process until for all blocks, there exists an integer $k_0$, such that $\Vert J_i^k\Vert_0=1,\enspace i=1,...,n$, for all $k\in K_t\subseteq K_{t-1}...\subseteq K_1,\enspace k\geqslant k_0$. Let $K':=K_t$. For all $i=1,...,n$, we define $x^*$ as follows:
\begin{equation}
(x^*)^{(i)}_{p_i}=
\begin{cases}
1,\enspace $if $p_i=J_i^k,\enspace k\in K',\enspace k\geqslant k_0,\\
0,\enspace $otherwise$.\notag
\end{cases}
\end{equation}
Then we find a global minimizer $x^*$ of (\ref{relaxation}) such that $\Vert x^*\Vert_0=n$. For $k\geqslant k_0,\enspace k\in K'$, there is $J_i^k=J_i^*,\enspace i=1,...,n$. Consequently, $x^*$ is also a global minimizer of (\ref{assignment}). Hence, (iii) holds. This completes the proof.
\end{proof}

\begin{proof}[Proof of Theorem~{\upshape\ref{theorem4}}]
Noting that $z$ is a global minimizer of (\ref{relaxation}), there is $\Gamma(z)\geqslant n$. Since 
$\lim_{k\to\infty}x_k=z$, there exists an integer $k'>0$ such that $(x_k)_p>\frac{1}{2},\enspace\forall p\in\Gamma(z),\enspace\forall k>k'.$ Therefore we have $\Gamma(z)\subseteq\Gamma^k$. Furthermore, there is $\Vert \Gamma(z)\Vert_0=n$ and $\Gamma(z)=\Gamma^k$ holds for all k $\geqslant k_1:=\max\left\{k_0,k'\right\}$. 
Consequently, $z$ is also a global minimizer of (\ref{assignment}), which completes the proof.
\end{proof}
\end{appendices}


\begin{thebibliography}{99}

\bibitem{allouche2014computational}
David Allouche, Isabelle Andr{\'e}, Sophie Barbe, Jessica Davies, Simon
  de~Givry, George Katsirelos, Barry O'Sullivan, Steve Prestwich, Thomas
  Schiex, and Seydou Traor{\'e}.
\newblock Computational protein design as an optimization problem.
\newblock {\em Artificial Intelligence}, 212:59--79, 2014.

\bibitem{anfinsen1973principles}
Christian~B Anfinsen.
\newblock Principles that govern the folding of protein chains.
\newblock {\em Science}, 181(4096):223--230, 1973.

\bibitem{bertsekas1982projected}
Dimitri~P Bertsekas.
\newblock Projected newton methods for optimization problems with simple
  constraints.
\newblock {\em SIAM Journal on control and Optimization}, 20(2):221--246, 1982.

\bibitem{boas2007potential}
F~Edward Boas and Pehr~B Harbury.
\newblock Potential energy functions for protein design.
\newblock {\em Current opinion in structural biology}, 17(2):199--204, 2007.

\bibitem{bomze1999maximum}
Immanuel~M Bomze, Marco Budinich, Panos~M Pardalos, and Marcello Pelillo.
\newblock The maximum clique problem.
\newblock {\em Handbook of Combinatorial Optimization: Supplement Volume A},
  pages 1--74, 1999.

\bibitem{chen2018outranking}
Ting-Yu Chen.
\newblock An outranking approach using a risk attitudinal assignment model
  involving pythagorean fuzzy information and its application to financial
  decision making.
\newblock {\em Applied Soft Computing}, 71:460--487, 2018.

\bibitem{chiu1998optimizing}
Ting~Lan Chiu and Richard~A Goldstein.
\newblock Optimizing potentials for the inverse protein folding problem.
\newblock {\em Protein engineering}, 11(9):749--752, 1998.

\bibitem{creighton1990protein}
Thomas~E Creighton.
\newblock Protein folding.
\newblock {\em Biochemical journal}, 270(1):1, 1990.

\bibitem{hypergraph}
Chunfeng Cui, Qingna Li, Liqun Qi, and Hong Yan.
\newblock A quadratic penalty method for hypergraph matching.
\newblock {\em Journal of Global Optimization}, 70(1):237--259, 2018.

\bibitem{desmet1992dead}
Johan Desmet, Marc~De Maeyer, Bart Hazes, and Ignace Lasters.
\newblock The dead-end elimination theorem and its use in protein side-chain
  positioning.
\newblock {\em Nature}, 356(6369):539--542, 1992.

\bibitem{florian2008application}
Michael Florian, Michael Mahut, and Nicolas Tremblay.
\newblock Application of a simulation-based dynamic traffic assignment model.
\newblock {\em European journal of operational research}, 189(3):1381--1392,
  2008.

\bibitem{forrester2008quadratic}
Richard~John Forrester and Harvey~J Greenberg.
\newblock Quadratic binary programming models in computational biology.
\newblock {\em Algorithmic Operations Research}, 3(2), 2008.

\bibitem{gainza2016algorithms}
Pablo Gainza, Hunter~M Nisonoff, and Bruce~R Donald.
\newblock Algorithms for protein design.
\newblock {\em Current opinion in structural biology}, 39:16--26, 2016.

\bibitem{gavish1978travelling}
Bezalel Gavish and Stephen~C Graves.
\newblock The travelling salesman problem and related problems.
\newblock 1978.

\bibitem{greenberg1969quadratic}
Harold Greenberg.
\newblock A quadratic assignment problem without column constraints.
\newblock {\em Naval Research Logistics Quarterly}, 16(3):417--421, 1969.

\bibitem{jaramillo2001automatic}
Alfonso Jaramillo, Lorenz Wernisch, Stephanie H{\'e}ry, and Shoshana~J Wodak.
\newblock Automatic procedures for protein design.
\newblock {\em Combinatorial chemistry \& high throughput screening},
  4(8):643--659, 2001.

\bibitem{jorge2006numerical}
Nocedal Jorge and J~Wright Stephen.
\newblock {\em Numerical optimization}.
\newblock Spinger, 2006.

\bibitem{lai2018real}
David~SW Lai and Janny~MY Leung.
\newblock Real-time rescheduling and disruption management for public transit.
\newblock {\em Transportmetrica B: Transport Dynamics}, 6(1):17--33, 2018.

\bibitem{lippow2007progress}
Shaun~M Lippow and Bruce Tidor.
\newblock Progress in computational protein design.
\newblock {\em Current opinion in biotechnology}, 18(4):305--311, 2007.

\bibitem{loiola2007survey}
Eliane~Maria Loiola, Nair Maria~Maia De~Abreu, Paulo~Oswaldo Boaventura-Netto,
  Peter Hahn, and Tania Querido.
\newblock A survey for the quadratic assignment problem.
\newblock {\em European journal of operational research}, 176(2):657--690,
  2007.

\bibitem{luo2017ccehc}
Chuan Luo, Shaowei Cai, Kaile Su, and Wenxuan Huang.
\newblock Ccehc: An efficient local search algorithm for weighted partial
  maximum satisfiability.
\newblock {\em Artificial Intelligence}, 243:26--44, 2017.

\bibitem{martello1990bin}
Silvano Martello and Paolo Toth.
\newblock Bin-packing problem.
\newblock {\em Knapsack problems: Algorithms and computer implementations},
  pages 221--245, 1990.

\bibitem{pabo1983molecular}
Carl Pabo.
\newblock Molecular technology: designing proteins and peptides.
\newblock {\em Nature}, 301(5897):200--200, 1983.

\bibitem{pierce2002protein}
Niles~A Pierce and Erik Winfree.
\newblock Protein design is np-hard.
\newblock {\em Protein engineering}, 15(10):779--782, 2002.

\bibitem{riazanov2017inverse}
Andrii Riazanov, Mikhail Karasikov, and Sergei Grudinin.
\newblock Inverse protein folding problem via quadratic programming.
\newblock {\em arXiv preprint arXiv:1701.00673}, 2017.

\bibitem{schiex2014computational}
Thomas Schiex.
\newblock Computational protein design as an optimization problem, 2014.

\bibitem{shah2004preprocessing}
Premal~S Shah, Geoffrey~K Hom, and Stephen~L Mayo.
\newblock Preprocessing of rotamers for protein design calculations.
\newblock {\em Journal of computational chemistry}, 25(14):1797--1800, 2004.

\bibitem{sun2006optimization}
Wenyu Sun and Ya-Xiang Yuan.
\newblock {\em Optimization theory and methods: nonlinear programming},
  volume~1.
\newblock Springer Science \& Business Media, 2006.

\bibitem{szeto2012dynamic}
Waiyuen Szeto and Sichun Wong.
\newblock Dynamic traffic assignment: model classifications and recent advances
  in travel choice principles.
\newblock {\em Central European Journal of Engineering}, 2:1--18, 2012.

\bibitem{thomas2008protein}
John Thomas, Naren Ramakrishnan, and Chris Bailey-Kellogg.
\newblock Protein design by sampling an undirected graphical model of residue
  constraints.
\newblock {\em IEEE/ACM Transactions on Computational Biology and
  Bioinformatics}, 6(3):506--516, 2008.

\bibitem{traore2013new}
Seydou Traor{\'e}, David Allouche, Isabelle Andr{\'e}, Simon De~Givry, George
  Katsirelos, Thomas Schiex, and Sophie Barbe.
\newblock A new framework for computational protein design through cost
  function network optimization.
\newblock {\em Bioinformatics}, 29(17):2129--2136, 2013.

\bibitem{xian2012application}
Ma~Xian-Ying.
\newblock Application of assignment model in pe human resources allocation.
\newblock {\em Energy procedia}, 16:1720--1723, 2012.

\bibitem{yanover2007dead}
Chen Yanover, Menachem Fromer, and Julia~M Shifman.
\newblock Dead-end elimination for multistate protein design.
\newblock {\em Journal of computational chemistry}, 28(13):2122--2129, 2007.

\bibitem{yue1992inverse}
Kaizhi Yue and Ken~A Dill.
\newblock Inverse protein folding problem: designing polymer sequences.
\newblock {\em Proceedings of the National Academy of Sciences},
  89(9):4163--4167, 1992.

\bibitem{mimo}
Ping-Fan Zhao, Qing-Na Li, Wei-Kun Chen, and Ya-Feng Liu.
\newblock An efficient quadratic programming relaxation based algorithm for
  large-scale mimo detection.
\newblock {\em SIAM Journal on Optimization}, 31(2):1519--1545, 2021.

\bibitem{zhu2007mixed}
Yushan Zhu.
\newblock Mixed-integer linear programming algorithm for a computational
  protein design problem.
\newblock {\em Industrial \& engineering chemistry research}, 46(3):839--845,
  2007.

\end{thebibliography}
\end{document}